\documentclass{amsart} % in was article
\usepackage{amssymb}
\usepackage{amsmath}
\usepackage{amsfonts}

\begin{document}
%%%%%%%%%%%%%%%%%%%%%%%%%%%%%%%%%%%%%%%%%%%%%%%%%%%%%%%%%%%%%%%%%%%%%%
%	spaces for your own definitions follows
%%%%%%%%%%%%%%%%%%%%%%%%%%%%%%%%%%%%%%%%%%%%%%%%%%%%%%%%%%%%%%%%%%%%%%
\newtheorem{theo}{Theorem}[section]
\newtheorem{atheo}{Theorem*}
\newtheorem{prop}[theo]{Proposition}
\newtheorem{aprop}[atheo]{Proposition*}
\newtheorem{lemma}[theo]{Lemma}
\newtheorem{alemma}[atheo]{Lemma*}
\newtheorem{exam}[theo]{Example}
\newtheorem{coro}[theo]{Corollary}
\theoremstyle{definition}
\newtheorem{defi}[theo]{Definition}
\newtheorem{rem}[theo]{Remark}

%\renewcommand{\theequation}{\mbox{\arabic{section}.\arabic{equation}}}

%letters - added these
\newcommand{\Bb}{{\bf B}}
\newcommand{\Cb}{{\mathbb C}}
\newcommand{\Nb}{{\mathbb N}}
\newcommand{\Qb}{{\mathbb Q}}
\newcommand{\Rb}{{\mathbb R}}
\newcommand{\Zb}{{\mathbb Z}}
\newcommand{\Ac}{{\mathcal A}}
\newcommand{\Bc}{{\mathcal B}}
\newcommand{\Cc}{{\mathcal C}}
\newcommand{\Dc}{{\mathcal D}}
\newcommand{\Fc}{{\mathcal F}}
\newcommand{\Ic}{{\mathcal I}}
\newcommand{\Jc}{{\mathcal J}}
\newcommand{\Kc}{{\mathcal K}}
\newcommand{\Lc}{{\mathcal L}}
\newcommand{\Oc}{{\mathcal O}}
\newcommand{\Pc}{{\mathcal P}}
\newcommand{\Sc}{{\mathcal S}}
\newcommand{\Tc}{{\mathcal T}}
\newcommand{\Uc}{{\mathcal U}}
\newcommand{\Vc}{{\mathcal V}}

\author{Nik Weaver}

\title [Paradoxes of rational agency]
       {Paradoxes of rational agency and formal systems that verify their own soundness}

\address {Department of Mathematics\\
Washington University\\
Saint Louis, MO 63130}

\email {nweaver@math.wustl.edu}

\date{\em Dec.\ 20, 2013}

\thanks{Partially supported by NSF grant DMS-1067726. I wish to thank the
participants of the December workshop on ``Reflection in Logic'' at MIRI
for a stimulating discussion on the topic of this paper.}

%%%%%%%%%%%%%%%%%%%%%%%%%%%%%%%%%%%%%%%%%%%%%%%%%%%%%%%%%%%%%%%%%%%%%%%
%	Please insert the article body now
%%%%%%%%%%%%%%%%%%%%%%%%%%%%%%%%%%%%%%%%%%%%%%%%%%%%%%%%%%%%%%%%%%%%%%%

\begin{abstract}
We consider extensions of Peano arithmetic which include an {\it assertibility}
predicate. Any such system which is arithmetically sound effectively
verifies its own soundness. This leads to the resolution of a range of
paradoxes involving rational agents who are licensed to act under
precisely defined conditions.
\end{abstract}

\maketitle

\section{Paradoxes of rational agency}

Let $S$ be a recursively axiomatized formal system that interprets
Peano arithmetic (PA). The soundness of $S$ is characterized by the scheme
$${\rm Prov}_S\ulcorner A(\bar{n}_1, \ldots, \bar{n}_k)\urcorner
\to A(n_1, \ldots, n_k)\eqno{(*)}$$
(``if $A$ is provable within $S$, then $A$''), with $A$ ranging over the
formulas of the language of $S$. Here $n_1, \ldots, n_k$ are the free
variables of $A$, $\bar{n}$ is the $n$th numeral, $\ulcorner A\urcorner$
is the
G\"odel number of $A$, and ${\rm Prov}_S\ulcorner A\urcorner$ is some standard
arithmetical formulation of the assertion that $A$ is a theorem of $S$.

The assertion that $S$ is consistent is effectively the special case of
($*$) when $A$ is any sentence which is provably false in $S$. Thus,
according to G\"odel's second incompleteness theorem, if $S$ is consistent
then it cannot prove any such instance of ($*$). Indeed, L\"ob's theorem
asserts that if $S$ is consistent then it can only prove
${\rm Prov}_S\ulcorner A\urcorner \to A$ when it can prove $A$.

The inability of a rational agent who reasons within $S$ to affirm the
soundness, or even the consistency, of $S$ is an unhappy but familiar
phenomenon. However, in a recent paper \cite{YH} Yudkowsky and Herreshoff
observe that this phenomenon carries a sharper sting in the context of an
AI that is licensed to act under precisely defined conditions. Thus,
imagine an intelligent machine $M$ that is capable of reasoning
within $S$ and is licensed to perform some action $\alpha_0$ when it has
verified the truth of some sentence $A_0$. We now describe a variety of
situations, mostly adapted from Section 3 of \cite{YH}, in which $M$ is
paradoxically unable to justify performing $\alpha_0$ even though it
intuitively ought to be able to do so.

\subsection{Naturalistic trust}

Suppose $M$ decides to improve its performance by constructing
an assistant $M'$ whose job is to prove theorems for $M$. $M$ programs the
assistant to reason within the same system $S$ that $M$ reasons in, and
to alert $M$ when it has proven the sentence $A_0$ (or any other sentence
which expresses an actionable criterion for $M$).

The paradox is that
even though $M$ knows that $M'$ reasons within $S$, it cannot act on
the information that $M'$ has proven $A_0$. Knowing this tells $M$ only
that $A_0$ is provable in $S$, not that $A_0$ is true, and lacking
the soundness scheme ($*$) it cannot infer $A_0$ from
${\rm Prov}_S\ulcorner A_0\urcorner$. In fact nothing $M'$ can tell
$M$, short of a line by line account of the actual proof, could
convince $M$ that $A_0$ is true. So apparently the best $M$
can do in this situation is to retrieve the formal proof of $A_0$
from $M'$ and check it line by line, thereby establishing $A_0$ to
$M$'s own satisfaction and licensing the action $\alpha_0$.

The puzzle here is that $M$ knows perfectly well what the outcome of
this verification is going to be, but still has no way to get around
executing it in its entirety. L\"ob's theorem prevents $M$ from
``trusting'' an agent in its environment, even granting that $M$ has
perfect knowledge that the agent reasons correctly within $S$.

\subsection{Reflective trust}

Alternatively, there can be situations in which it might be relatively easy
for $M$ to prove the statement ${\rm Prov}_S\ulcorner A_0\urcorner$, and
even to produce an algorithm which provably generates a formal proof of $A_0$,
but unfeasably difficult for $M$ to actually execute that algorithm. The
formal proof of $A_0$ might be astronomically long, for example, even while
the proof that it can be constructed is quite short. Again, $M$ finds itself
in a situation where it ``knows'' that it can prove $A_0$, but is unable to
act on this knowledge without first performing some tedious or even
unfeasable computation
whose result is known in advance. Evidently $M$ cannot even trust itself.

The paradoxes of naturalistic and reflective trust are both straightforward
expressions of $M$'s inability to affirm the soundness scheme ($*$),
and both of them could be easily handled by modifying its licensing
criteria. All we have to do is to program $M$ so that whenever it
accepts a sentence $A$ as a license to perform the action $\alpha$, it
also accepts ${\rm Prov}_S\ulcorner A\urcorner$ as a license to perform
$\alpha$. This would allow it to act on the knowledge that $A$ is provable
--- knowledge that it might derive on its own or obtain from an outside
source --- without actually possessing a proof of $A$.
Next we present two additional difficulties which cannot be
handled in this way.

\subsection{Reflectively coherent trust}

Suppose the actionable condition has the form $(\forall n)A_0(n)$. It
could be the case that $M$ is able to prove each instance $A_0(\bar{n})$,
yet not able to prove the
quantified statement $(\forall n)A_0(n)$.

So far, there is nothing worrisome about this possibility. It could easily
happen that each instance $A_0(\bar{n})$ is verifiable by a finite
computation, yet there is no uniform reason why $A_0(n)$ is true for every
$n$. (Think of assertions like ``the first through $n$th digits in the
decimal expansion of $\pi$ do not contain a string of 100 consecutive 9's''.)
But suppose in addition that $M$ knows that for each $n$ it can prove
$A_0(\bar{n})$. That is, suppose $M$ has proven the sentence
$$(\forall n){\rm Prov}_S\ulcorner A_0(\bar{n})\urcorner.$$
Without knowing that $S$ is sound, $M$ cannot go on to infer the
condition $(\forall n)A_0(n)$ which allows it to act.

Yudkowsky and Herreshoff ask whether it is possible to design the
formal system $S$ in a way that evades this problem. That is, can
$S$ have the property that whenever
$(\forall n){\rm Prov}_S\ulcorner A(\bar{n})\urcorner$ is a theorem of
$S$, it is also the case that $(\forall n)A(n)$ is a theorem of $S$?

Unfortunately, the answer is no. No consistent recursively axiomatized
system which interprets Peano arithmetic is reflectively coherent in
this sense. For let $S$ be any recursively axiomatized system which
interprets PA. We claim that we can prove, in $S$, the statement
$(\forall g){\rm Prov}_S\ulcorner A(\bar{g})\urcorner$ where $A(g)$ is
a standard arithmetization of the assertion ``$g$ is not the G\"odel
number of a proof in $S$ of $0 = 1$''. This is done as follows. Reasoning
informally in $S$, we argue that for each $g$, either $g$ is not a proof
of $0 = 1$, in which case this fact can be verified by a mechanical finite
computation and hence is trivially provable in $S$, or else $g$ is a proof
of $0 = 1$, in which case $S$ is inconsistent and therefore proves anything.
(This argument appears in \cite{F}, Section 2 (d).) Thus, reflective
coherence would entail that $S$ can prove its own consistency, contradicting
G\"odel's second incompleteness theorem as improved by Rosser.

However, it still seems reasonable to expect $M$ to be able to accept a
proof of $(\forall n){\rm Prov}_S\ulcorner A(\bar{n})\urcorner$ as
licensing the same actions that are licensed by a proof of $(\forall n)A(n)$.
The point here is that agreeing to accept
${\rm Prov}_S\ulcorner (\forall n)A(n)\urcorner$ as a licensing condition,
as proposed above, does not accomplish this.
We should also note that the issue is no longer about feasability:
it may simply be impossible for $M$ to convince itself that
$(\forall n)A(n)$ holds, despite knowing that
$(\forall n){\rm Prov}_S\ulcorner A(\bar{n})\urcorner$.

\subsection{Disjunctive trust}

Finally, suppose the actionable condition is the sentence $A_0$
and we follow the suggestion made above to program $M$ to also accept
${\rm Prov}_S\ulcorner A_0\urcorner$,
${\rm Prov}_S\ulcorner{\rm Prov}_S\ulcorner A_0\urcorner\urcorner$,
etc., as licensing the same action. This does not accomodate the possibility
that $M$ might prove the sentence
$A_0 \vee {\rm Prov}_S\ulcorner A_0\urcorner$. (This formulation of the
disjunctive trust problem was suggested to me by Cameron Freer.)
Intuitively, the preceding sentence tells us
that either $A_0$ is true, which licenses action, or else $A_0$ is
provable and therefore true, which also licenses action. So
$A_0 \vee {\rm Prov}_S\ulcorner A_0\urcorner$ ``should'' license action,
but it does not. We cannot straightforwardly infer either $A_0$
or ${\rm Prov}_S^k\ulcorner A_0\urcorner$ for any $k$.

A natural idea is to augment $S$ with a truth predicate. This would
enable us to formalize the reasoning we just used which allowed
us to infer the truth of $A_0$ from the provability of
$A_0 \vee {\rm Prov}_S\ulcorner A_0\urcorner$. The problem is that
self-applicative truth predicates are inconsistent, whereas a non
self-applicative truth predicate would only apply to reasoning carried
out in the original system, not the augmented system. Thus $M$ could still
find itself in a situation where it has proven that it can prove $A_0$,
but is unable to infer $A_0$ because the anticipated proof of $A_0$ makes
use of the truth predicate. A partially self-applicative truth predicate
a la Kripke would not do any better; by L\"ob's theorem it is simply
impossible to consistently augment $S$ in any way that would enable us
to generally infer $A$ from ${\rm Prov}_{S^+}\ulcorner A\urcorner$, where
${\rm Prov}_{S^+}$ refers to provability in the augmented system $S^+$.

As we have seen,
the paradoxes of reflective coherence and disjunctive trust are not resolved
by broadening the licensing criteria so that ${\rm Prov}_S\ulcorner A\urcorner$
licenses any action that $A$ licenses. There can still be provable
assertions which ought to license actions but do not. Another possibility
is to simply ignore the provability predicate for licensing purposes;
that is, program $M$ so that whenever $A$ licenses $\alpha$, so does
any sentence which reduces to $A$ when all provability predicates are
removed. But this idea is a non-starter because provably true sentences can
become false when provability predicates are removed. For instance, we
can prove in PA that
$${\rm Prov}_{\rm PA}\ulcorner{\rm Con}({\rm PA})\urcorner
\to {\rm Prov}_{\rm PA}\ulcorner 0=1\urcorner$$
(if PA proves its own consistency, then it is inconsistent)
but after removing the provability predicates this becomes
$${\rm Con}({\rm PA}) \to 0=1,$$
i.e., $\neg{\rm Con}({\rm PA})$.

What we need is a principled method of broadening some initially given
licensing criteria that goes beyond merely accepting
${\rm Prov}\ulcorner A\urcorner$ in place of $A$.

\section{Formalizing assertibility}

We propose to handle the paradoxes of rational agency using the notion of
{\it assertibility} --- more precisely, {\it rational} or
{\it warranted} assertibility. This is a philosophical term which is
supposed to identify a property of sentences that expresses our right
to assert them. The idea is that a sentence is assertible if it has a
perfect rational justification or ``warrant''. In a mathematical context
the warrant may be thought of as a proof, in the semantic sense of
{\it argument that provides a perfect rational justification}, not in
the syntactic sense of {\it formal proof within some formal system}. In
intuitionism the word ``provability'' is used in the
former sense, and thus for our purposes is synonymous with assertibility.

One could object that this notion of assertibility or provability is either
altogether meaningless, or at best too vague to support precise analysis. A
philosophical argument can be made against this objection, but perhaps the
best answer is to simply exhibit a formal treatment of assertibility. An
axiomatization was given in \cite{W1}; it goes as follows.
Let $\Box$ be a predicate symbol which represents ``is assertible'' and
is to be applied to G\"odel numbers of sentences in some language
(which, crucially, might itself employ the symbol $\Box$).
There are five axioms for $\Box$,
\begin{enumerate}
\item $\Box\ulcorner A \vee B\urcorner \qquad\leftrightarrow\qquad
\Box \ulcorner A \urcorner \vee \Box \ulcorner B\urcorner$

\item $\Box\ulcorner A \wedge B\urcorner \qquad\leftrightarrow\qquad
\Box \ulcorner A\urcorner \wedge \Box \ulcorner B\urcorner$

\item $\Box\ulcorner (\exists n) A(n)\urcorner \qquad\leftarrow\qquad
(\exists n) \Box \ulcorner A(\bar{n})\urcorner$

\item $\Box\ulcorner(\forall n) A(n)\urcorner \qquad\leftrightarrow\qquad
(\forall n) \Box \ulcorner A(\bar{n})\urcorner$

\item $\Box\ulcorner A \to B\urcorner \qquad\to\qquad
\Box \ulcorner A\urcorner \to \Box \ulcorner B\urcorner$,
\end{enumerate}
and one axiom scheme, {\it capture}, which states
\begin{enumerate}
\setcounter{enumi}{5}
\item $A(n_1, \ldots, n_k)\qquad\to\qquad
\Box\ulcorner A(\bar{n}_1, \ldots, \bar{n}_k)\urcorner$
\end{enumerate}
for every formula $A$. Intuitionistic logic is to be used; we do not
assume that the law of excluded middle holds for formulas which involve $\Box$.

There is no separate axiom for negation. We treat negation as a derived
symbol, such that $\neg A$ is an abbreviation of $A \to \perp$ where $\perp$
is some canonical falsehood such as $0=1$.

The justification for these axioms is discussed in detail in \cite{W1}.
The main points are that the logic is intuitionistic and that the
{\it release} principle $\Box\ulcorner A\urcorner \to A$ is not included.
Excluded middle is suspect because
we cannot a priori assign truth values to all statements of the form
$\Box\ulcorner A\urcorner$: if putative proofs can refer to each other,
such that the validity of one hinges on the validity of the other, then
there is a potential circularity issue which could make the assignment of
truth values problematic. Similarly, whenever we have
actually proven that a sentence
$A$ is provable we ought to be able to rationally infer $A$,
but the implication $\Box \ulcorner A\urcorner \to A$ is suspect because it
globally affirms the validity of all proofs, including proofs in which it
itself might have been used, creating another circularity issue.

(The capture scheme is constructively valid because we can simply
stipulate that any proof of $A$ must be recognizable as a proof. Thus,
any proof of $A$ can trivially be converted into a proof that $A$ is
provable.)

A surprising feature of the $\Box$ operator is that it can be applied
self-referentially without producing a contradiction. For example, consider
a sentence $L$ that says of itself that it is not assertible. Thus, we have
$L \equiv \neg \Box \ulcorner L\urcorner$. Then we
can make the following deductions. First, assuming $L$ lets us
immediately infer $\neg\Box\ulcorner L\urcorner$, and it also lets
us infer $\Box \ulcorner L\urcorner$ via capture.
Thus $L$ entails a statement and its negation, so
we may conclude $\neg L$. We can further deduce $\Box\ulcorner\neg L\urcorner$.
However, there is no contradiction here. In particular, we cannot proceed to
infer $\neg \Box \ulcorner L\urcorner$, only the weaker statement
$\Box\ulcorner L\urcorner \to \Box \ulcorner\perp\urcorner$.

We do not merely claim that simple attempts to derive a contradiction
fail. We can actually give consistency proofs for systems which allow the
formulation of assertibility versions of the liar paradox and Russell's
paradox (\cite{W1} and \cite{W2}; see also \cite{W3}). See also Theorem
\ref{consistent} below.

It is interesting to note that including either the law of excluded middle or
a release axiom scheme would make assertibility reasoning paradoxical.
If we knew $\Box \ulcorner L \urcorner \to L$, then together with $\neg L$,
which we just proved, we could infer $\neg \Box \ulcorner L\urcorner$. But
that is equivalent to $L$, so we would have proven both $L$ and $\neg L$, a
contradiction. Or again,
if we knew $\Box\ulcorner L\urcorner \vee \neg \Box \ulcorner L\urcorner$
then we could reason as follows. First, assume $\Box\ulcorner L\urcorner$.
Since we have already proven
$\neg L$, we also have $\Box\ulcorner \neg L\urcorner$.
This yields $\Box \ulcorner L \wedge \neg L\urcorner$, and a short
argument then gives us $\Box \ulcorner\perp\urcorner$. On the other
hand, if we assume $\neg \Box\ulcorner L\urcorner$ then, as above,
we can infer $\perp$ and from this $\Box\ulcorner \perp\urcorner$.
Thus, if we had the law of excluded middle then we could prove
$\Box\ulcorner \perp\urcorner$. So excluded middle forces us to
affirm a contradiction.

\section{Self-verifying systems}

In what follows we employ intuitionistic logic and take
$\perp$ to be the formula $0 = 1$. Systems using classical
logic can be accomodated by including all instances of the law of excluded
middle as non-logical
axioms. Thus, for example, we treat Peano arithmetic as an
intuitionistic system but include as non-logical axioms all formulas of the
form $A \vee \neg A$ with $A$ a formula of first order arithmetic.

It is convenient to exclude formulas with free variables from our formal
proofs. Thus we require all axioms to be sentences and we express the
generalization deduction rules as the implications
$$(\forall n_1, \ldots, n_k)(A \to B) \to
(\forall n_2, \ldots, n_k)(A \to (\forall n_1)B)$$
(whenever $n_1$ is not free in $A$) and
$$(\forall n_1, \ldots, n_k)(A \to B) \to
(\forall n_2, \ldots, n_k)((\exists n_1)A \to B)$$
(whenever $n_1$ is not free in $B$).
The only deduction rule we need then is modus ponens in the form
$${\rm given}\quad(\forall n_1, \ldots, n_k)A\quad{\rm and}\quad
(\forall n_1, \ldots, n_k)(A \to B),\quad{\rm infer}\quad
(\forall n_1, \ldots, n_k)B.$$

In this section we assume
$S$ is a recursively axiomatized theory in the language of first
order arithmetic which extends PA. Define a formal system $S_\Box$ as follows.
Its language is the language of first order arithmetic enriched by a single
unary relation symbol $\Box$. Fix a G\"odel numbering for this language.
The {\it main} axioms of $S_\Box$ consist of the universal closures of
\begin{itemize}
\item the non-logical axioms of $S$,

\item all instances of the induction scheme and the logical axiom schemes
for formulas in the language of $S_\Box$, and

\item all of the axioms (1) -- (6) for $\Box$ given in the last section,
for the language of $S_\Box$.
\end{itemize}
(For details on how one would formalize, e.g., axioms (3) and (4) for
$\Box$, see Section 2 (c) of \cite{F}.)
Also, let ${\rm Ax}(g)$ be a formula in the language of first order arithmetic
which expresses that $g$ is the G\"odel number of one of the main axioms.
We can assume that whenever ${\rm Ax}(\bar{g})$ holds this is provable in PA.
Besides its main axioms, $S_\Box$ has one additional {\it jump} axiom given
by the formula
$$(\forall g)({\rm Ax}(g) \to \Box(g)).$$

Let ${\rm Prov}_{S_\Box}(g)$ be a standardly expressed formula stating that
$g$ is the G\"odel number of a theorem of $S_\Box$. (Recall that in our setup
every theorem is a sentence.) We now show that $S_\Box$ verifies its own
soundness in the sense that it proves a single statement which affirms that
every theorem of $S_\Box$ is assertible.

\begin{theo}\label{sound}
The sentences
$$(\forall g)({\rm Prov}_{S_\Box}(g) \to \Box(g))$$
(assertible soundness) and
$$\Box\ulcorner{\rm Con}(S_\Box)\urcorner$$
(assertible consistency) are provable in $S_\Box$.
\end{theo}

\begin{proof}
Working in $S_\Box$, we know from the jump axiom that we have $\Box(g)$
whenever $g$ is the G\"odel number of one of the main axioms.  Also, if
$B$ is the jump axiom itself then we can use capture to infer
$\Box\ulcorner B\urcorner$. So we can show, in $S_\Box$, that $\Box(g)$
holds for any $g$ which is the G\"odel number of any of the axioms of $S_\Box$.

Still working in $S_\Box$, we proceed to prove $\Box(g)$ whenever $g$ is
the G\"odel number of any sentence that is provable in $S_\Box$.
Given any proof of such a sentence, we inductively verify $\Box(g')$ as
$g'$ ranges over the G\"odel numbers of the
lines of the proof. Given the result of the preceding paragraph,
we just have to show how to handle deduction via universally quantified
modus ponens. We do this with deductions of the form
\medskip

{\narrower{
\noindent
given $\Box\ulcorner (\forall n)A(n)\urcorner$,
$\Box\ulcorner (\forall n)(A(n) \to B(n))\urcorner$

\noindent
infer $(\forall n)\Box\ulcorner A(\bar{n})\urcorner$,
$(\forall n)\Box\ulcorner A(\bar{n}) \to B(\bar{n})\urcorner$

\noindent
infer $(\forall n)(\Box\ulcorner A(\bar{n})\urcorner \wedge
\Box\ulcorner A(\bar{n}) \to B(\bar{n})\urcorner)$

\noindent
infer $(\forall n)\Box\ulcorner B(\bar{n})\urcorner$

infer $\Box\ulcorner(\forall n)B(n)\urcorner$}}
\medskip

\noindent
(assuming here only one universal quantifier for the sake of notational
simplicity). This completes the proof of assertible soundness. For
assertible consistency, recall from Section 1.3 that $S_\Box$ proves the
statement $(\forall g){\rm Prov}_{S_\Box}\ulcorner A(\bar{g})\urcorner$
where $A(g)$ arithmetizes the assertion that $g$ is not the G\"odel number
of a proof in $S_\Box$ of $0=1$. By assertible soundness we
can infer $(\forall g)\Box\ulcorner A(\bar{g})\urcorner$, and then using
box axiom (4) we can infer $\Box\ulcorner (\forall g)A(g)\urcorner$, i.e.,
$\Box\ulcorner{\rm Con}(S_\Box)\urcorner$.
\end{proof}

Using G\"odelian self-reference techniques it is not hard to write down
a sentence in the language of $S_\Box$ which says of itself that it is
not assertible. But as we discussed earlier, no contradiction results.
We will now prove that $S_\Box$ is consistent, provided $S$ is sound
(i.e., the axioms of $S$ are true statements of first order arithmetic).

The systems described in \cite{W1} and \cite{W2} also included the release
principle in the form of a deduction rule which allows the inference of $A$
from $\Box\ulcorner A\urcorner$. The justification for this rule is
that whenever we have actually proven that $A$ is assertible we should be
entitled to assert $A$. Thus, having accepted a system that does not employ
the release rule we can successively accept proofs that employ one use of
the rule, then proofs that employ two uses, and so on.

In the present paper
we exclude the release rule because it complicates Theorem \ref{sound}:
${\rm Prov}_{S_\Box}\ulcorner A\urcorner$ would no longer imply
$\Box\ulcorner A\urcorner$, it would imply $\Box^{k+1}\ulcorner A\urcorner$
where $k$ is the number of uses of the release rule in a proof of $A$.
However, the point stands that proofs of not only $A$, but also
$\Box\ulcorner A\urcorner$, $\Box\ulcorner\Box\ulcorner A\urcorner\urcorner$,
$\ldots$, all affirm ``trust'' in the semantic content of $A$. Thus a
consistency result for $S_\Box$ should not only show that $0 = 1$ is
unprovable, it should show that $\Box^{k+1}\ulcorner 0=1\urcorner
\equiv \Box\ulcorner\Box\ulcorner \cdots 0=1\cdots \urcorner\urcorner$
($k+1$ terms) is unprovable for all $k$. This is what we establish now.
The argument is similar to the proofs of Theorem 6.1
of \cite{W1} and Theorem 5.1 of \cite{W2}.

\begin{theo}\label{consistent}
If $S$ is sound and $A$ is a false sentence of first order arithmetic, then
$A$ is not a theorem of $S_\Box$, nor is $\Box^k\ulcorner A\urcorner$ for
any $k \geq 1$.
\end{theo}

\begin{proof}
We define a sequence $(F_i)$ such that each $F_i$ is a set of sentences
in the language of $S_\Box$. Intuitively, these are sentences that we
determine to be false. The definition proceeds by recursion on $i$,
and for a given value of $i$ by recursion on the complexity of a sentence. 
Fixing $i$, we define $F_i$ as follows. The atomic sentences have the
form $t = t'$ and $\Box(t)$ where $t$ and $t'$ are numerical terms
(i.e., they contain no variables). Put $t = t'$ in $F_i$ if $t$ and $t'$
numerically evaluate to different numbers, and put $\Box(t)$ in $F_i$
if $t$ evaluates to the G\"odel number of a sentence that belongs to $F_{i-1}$.
We do not place any sentence of the form $\Box(t)$ in $F_0$.

Place $A \wedge B$ in $F_i$ if either $A$ or $B$ belongs to $F_i$; place
$A \vee B$ in $F_i$ if both $A$ and $B$ belong to $F_i$; place
$(\forall n)A(n)$ in $F_i$ if $A(\bar{n})$ belongs to $F_i$ for some $n$;
place $(\exists n)A(n)$ in $F_i$ if $A(\bar{n})$ belongs to $F_i$ for all $n$.
Place $A \to B$ in $F_i$ if for some $j \leq i$ we have $A \not\in F_j$
and $B \in F_j$.

It is easy to see that $F_i \subseteq F_{i+1}$ for all $i$. Let
$F = \bigcup_i F_i$. It is tedious but straightforward to verify both that
no axiom of $S_\Box$ belongs to $F$ and that the complement of $F$ is stable
under universally quantified modus ponens. Thus no theorem of $S_\Box$ lies
in $F$. But every false sentence of first order arithmetic belongs to
$F_0$, so $\Box^k\ulcorner A\urcorner$ belongs to $F_k$ for every $k \geq 1$.
So none of these statements can be a theorem of $S_\Box$.
\end{proof}

It is easy to see that consistency of $S$ implies consistency of $S_\Box$;
any model of $S$ can be extended to a model of $S_\Box$ by letting
$\Box(g)$ hold for all $g$. However,
mere consistency of $S$ is not sufficient to guarantee that
$\Box\ulcorner 0=1\urcorner$ is unprovable in $S_\Box$. For example, take
$S$ to be ${\rm PA} + \neg{\rm Con}({\rm PA})$.
Then Theorem \ref{sound} shows that $S_\Box$ proves
$${\rm Prov}_{S_\Box}\ulcorner 0 =1\urcorner \to \Box\ulcorner 0=1\urcorner,$$
but we can also infer ${\rm Prov}_{S_\Box}\ulcorner 0=1\urcorner$ from
$\neg{\rm Con}({\rm PA}) \equiv
{\rm Prov}_{\rm PA}\ulcorner 0 =1\urcorner$,
so that $\Box\ulcorner 0=1\urcorner$ is a theorem of $S_\Box$.

\section{Resolving the paradoxes}

We propose the following uniform resolution of the paradoxes of rational
agency discussed in Section 1. First, we require agents to reason
within systems that have an assertibility predicate satisfying the axioms
for $\Box$ and for which we are
able to prove a version of Theorem \ref{sound}. This addresses their inability
to affirm the consistency and soundness of their own reasoning.
We also impose the following licensing rule.
\begin{quote}
{\bf Box rule:} Whenever a sentence $A$ licenses some action $\alpha$,
the sentence $\Box\ulcorner A\urcorner$ also licenses $\alpha$.
\end{quote}

The resolution of the paradoxes is now simple and straightforward. Whenever
${\rm Prov}_{S_\Box}\ulcorner A\urcorner$ is a theorem, so is
$\Box\ulcorner A\urcorner$. Therefore, according to the box rule, knowing
that $A$ is provable is always just as actionable as knowing $A$. This
handles naturalistic and reflective trust. For
reflectively coherent trust, observe that
$$(\forall n){\rm Prov}_{S_\Box}\ulcorner A(\bar{n})\urcorner \to
(\forall n)\Box\ulcorner A(\bar{n})\urcorner \to
\Box\ulcorner(\forall n)A(n)\urcorner,$$
provably in $S_\Box$.
So $(\forall n){\rm Prov}_{S_\Box}\ulcorner A(\bar{n})\urcorner$ is
just as actionable as $(\forall n)A(n)$. Finally, disjunctive trust is
handled by the inference
$$(A \vee {\rm Prov}_{S_\Box}\ulcorner A\urcorner) \to
(A \vee \Box\ulcorner A\urcorner) \to
(\Box\ulcorner A\urcorner \vee \Box\ulcorner A\urcorner)\to
\Box\ulcorner A\urcorner,$$
which shows that $A \vee {\rm Prov}_{S_\Box}\ulcorner A\urcorner$
is just as actionable as $A$.

Do we get too much? For instance, is
${\rm Prov}_{\rm PA}\ulcorner{\rm Con}({\rm PA})\urcorner
\to {\rm Prov}_{\rm PA}\ulcorner 0=1\urcorner$ just as actionable as
${\rm Con}({\rm PA}) \to 0=1$? No, because the implication
${\rm Prov}_{S_\Box}\ulcorner A\urcorner \to \Box\ulcorner A\urcorner$
only goes in one direction, and besides, we cannot bring an implication
inside the box operator. So there is no way to remove the provability
predicates in this case.

We do not mean to imply that rational agents should be required to
work in the language of first order arithmetic. That is merely a
convenient vehicle for the results we proved in Section 3, but all
kinds of formal systems are amenable to augmentation by an
assertibility operator. For instance,
a version of Theorem \ref{consistent} for ${\rm ZFC}_\Box$ should be
provable under the assumption that inaccessible cardinals exist.

\section{Self-modifying AI}

We have not yet addressed the main concern of \cite{YH}, which involves a
rational agent who is seeking not a license to perform a particular
action, but general
permission to delegate the performance of actions to a second agent. This
is more difficult because the first agent does not merely have to sanction
the second agent's judgement that some particular action has been licensed,
as in the naturalistic trust paradox. Rather, the first agent is required
to globally affirm the correctness of any such judgement the second agent
might make.

The same issues appear in the case of an intelligent machine that is
considering modifying its own source code (in order to make itself
more intelligent, say). Before doing this it would want to be sure
that its post-modification state will reason correctly, i.e., any
theorem it proves after the modification should actually be true. This
runs into the familiar L\"obian difficulty that the agent is not even
able to affirm the soundness of its pre-modification reasoning.

One way to deal with this problem is to forbid the second agent from
performing any action, reducing it to the role of a theorem proving
assistant. By allowing the second agent only to think, not to act, we
effectively convert the problem into a type of naturalistic trust question
that we already know how to handle. In the self-modification scenario
this could amount to putting a restriction on which parts of the
machine's source code it is allowed to rewrite.

If one insists on allowing the second agent to act the problem becomes much
harder. In Section 4 of \cite{YH}, two constructions are presented of an
infinite sequence of independently acting
agents, each of whom can give a provable justification for
activating the next one, yet none of whose deductive power falls below
an initially prescribed level. The constructions are clever but they
have a nonstandard flavor. Probably this is unavoidable, unless the
problem description were to be altered in some fundamental way. In the
remainder of this section we present another solution which uses
nonstandard assertibility reasoning.

We begin by setting up some formalism. Enumerate the possible actions the
first agent $M_1$ can take. Let $S$ be a base system
that uses the language of first order arithmetic augmented by some
additional relation symbols. We assume that using these extra symbols it
is possible to formulate an expression ${\rm Act}_{M_1}(n)$ which
represents the event that ${M_1}$ takes the $n$th action. Also let $\Gamma$ be
a propositional symbol (i.e., a nullary relation symbol) which represents
the event that some desired goal $G$ is achieved. Next, we augment $S$ by
introducing an infinite sequence of constant symbols $(\kappa_i)$ together
with, for each $i$, the axiom $\kappa_i = \kappa_{i+1}+1$. Call the resulting
system $S^*$. It is clear that if $S$ is consistent, then so is any finite
fragment of $S^*$, and therefore so is $S^*$.

Although its language is slightly
richer than the language used in Section 3, we can still augment $S^*$ with
an assertibility predicate to get a new system $S_\Box^*$ in the manner
described there. We assume that versions of Theorems \ref{sound} and
\ref{consistent} hold for $S_\Box^*$.

Following \cite{YH}, we impose a condition that $M_1$ only acts if it can
prove that doing so will achieve $G$. This licensing criterion could be
formalized in $S$ by the formula ${\rm Act}_{M_1}(n) \to \Gamma$,
but we weaken it by inserting an assertibility operator to require
$${\rm Act}_{M_1}(n) \to \Box^{\kappa_1}\ulcorner\Gamma\urcorner\eqno{(\dag)}$$
for the first agent,
$${\rm Act}_{M_2}(n) \to \Box^{\kappa_2}\ulcorner\Gamma\urcorner$$
for its successor agent, and so on. Thus the condition which describes
when $M_1$ acts is
$$(\forall n)({\rm Act}_{M_1}(n) \to {\rm Prov}_{S_\Box^*}\ulcorner{\rm Act}_{M_1}(\bar{n}) \to \Box^{\kappa_1}\ulcorner\Gamma\urcorner\urcorner).$$

Now assume that $M_1$ can prove the corresponding condition
$$(\forall n)({\rm Act}_{M_2}(n) \to {\rm Prov}_{S_\Box^*}\ulcorner{\rm Act}_{M_2}(\bar{n}) \to \Box^{\kappa_2}\ulcorner\Gamma\urcorner\urcorner)\eqno{(1)}$$
that describes when $M_2$ acts. Since $M_1$ will not activate $M_2$ unless it
can trust that $M_2$ will achieve $G$, we also assume that $M_1$ can prove
$${\rm Act}_{M_1}(\bar{n}_0) \to (\exists n){\rm Act}_{M_2}(n)\eqno{(2)}$$
where $\bar{n}_0$ indexes the action of activating $M_2$; that
is, $M_1$ knows that if $M_2$ is activated then it will act in
some way. (This could be trivial if we know there is a brute force method of
achieving $G$ and want $M_2$ to look for an efficient method. Alternatively,
we could alter the definition of $G$ so that, say, an exhaustive search which
fails to find any way to accomplish the desired goal, counts as an acceptable
outcome.)

In order to license activating $M_2$, we must prove the sentence
${\rm Act}_{M_1}(\bar{n}_0) \to \Box^{\kappa_1}\ulcorner\Gamma\urcorner$.
This can be done as follows. Working in $S_\Box^*$, combine (1) and (2) to get
$${\rm Act}_{M_1}(\bar{n}_0) \to (\exists n)[{\rm Act}_{M_2}(n) \wedge
{\rm Prov}_{S_\Box^*}\ulcorner{\rm Act}_{M_2}(\bar{n}) \to \Box^{\kappa_2}\ulcorner\Gamma\urcorner\urcorner].$$
Theorem \ref{sound} then yields
$${\rm Act}_{M_1}(\bar{n}_0) \to (\exists n)[{\rm Act}_{M_2}(n) \wedge
\Box\ulcorner{\rm Act}_{M_2}(\bar{n}) \to \Box^{\kappa_2}\ulcorner\Gamma\urcorner\urcorner]$$
which simplifies to
$${\rm Act}_{M_1}(\bar{n}_0) \to \Box\ulcorner(\exists n)[{\rm Act}_{M_2}(n)
\wedge ({\rm Act}_{M_2}(n) \to \Box^{\kappa_2}\ulcorner\Gamma\urcorner)]\urcorner$$
then to
$${\rm Act}_{M_1}(\bar{n}_0) \to \Box\ulcorner\Box^{\kappa_2}\ulcorner\Gamma\urcorner\urcorner$$
and finally to
$${\rm Act}_{M_1}(\bar{n}_0) \to \Box^{\kappa_1}\ulcorner\Gamma\urcorner$$
as desired. Thus $M_1$ is licensed to activate $M_2$. Similarly,
$M_2$ can be licensed to activate a third agent $M_3$, and so on.
All of these agents reason within the same formal system $S_\Box^*$, but
they respond to different licensing criteria.

The licensing criteria appear to become stricter for later agents, which is
why each $M_i$ is able to accept the validity of actions performed on the
basis of $M_{i+1}$'s licensing criteria. However, since the $\kappa_i$ do
not have definite values the criteria for $M_i$ and $M_{i+1}$ are effectively
equivalent.

A possible intuition is that $\kappa_1$ is some unspecified very large
number. This seems like a fairly reasonable intuition for a formal system
that is admittedly nonstandard. As we mentioned earlier, it seems unlikely
that there is any standard solution to the problem of delegated action as
it is currently posed.

The licensing criteria for the $M_i$ do not respect the box rule. That
could be accomodated by introducing transfinite degrees of assertibility
and accepting
$$\Box^j\ulcorner{\rm Act}_{M_i}(\bar{n}) \to \Box^{\kappa_i\cdot\omega}\ulcorner\Gamma\urcorner\ulcorner$$
for any value of $j$ as a license for $M_i$ to execute the $n$th action.
But we do not pursue this direction.

%##

\end{document}